\documentclass[conference]{IEEEtran}
\IEEEoverridecommandlockouts
\usepackage{color}
\usepackage{theoremref, hyperref}
\usepackage{amsmath, amsthm}
\usepackage{amssymb}

\numberwithin{equation}{section}

\definecolor{darkcyan}{rgb}{0.0, 0.55, 0.55}

\newcommand{\Hil}[0]{
\mathcal{H}
}

\newcommand{\B}[0]{{\mathcal{B}}}

\newtheorem{thm}{Theorem}[section]
\newtheorem{defin}[thm]{Definition}
\newtheorem{prop}[thm]{Proposition}
\newtheorem{lem}[thm]{Lemma}
\newtheorem{cor}[thm]{Corollary}
\newtheorem{rem}[thm]{Remark}

\newtheorem{conj.}[thm]{Conjecture}
\newtheorem{Bsp.}[thm]{Example}

\newcommand{\Hc}{\mathcal H}

\newcommand{\Bc}{\mathcal B}
\newcommand{\Tc}{\mathcal T}
\newcommand{\norm}[1]{\left\|#1\right\|}

\newcommand{\limn}{\lim_{n \to \infty}}
\newcommand{\jbr}[1]{\langle#1\rangle}
\newcommand{\ljbr}[1]{\left\langle#1\right\rangle}
\newcommand{\spanl}{\operatorname{span}}

\newcommand{\brac}[1]{\left(#1\right)}
\newcommand{\pw}{\text{pw}}
\newcommand{\id}{\text{id}}
\newcommand{\abs}[1]{\left|#1\right|}

\newcommand{\R}{\mathbb R}
\newcommand{\N}{\mathbb N}
\newcommand{\C}{\mathbb C}

\usepackage{cite}
\usepackage{amsmath,amssymb,amsfonts}
\usepackage{algorithmic}
\usepackage{graphicx}
\usepackage{textcomp}
\usepackage{xcolor}
\def\BibTeX{{\rm B\kern-.05em{\sc i\kern-.025em b}\kern-.08em
    T\kern-.1667em\lower.7ex\hbox{E}\kern-.125emX}}
\begin{document}

\title{Details on the distribution co-orbit space $\Hil^{\infty}_w$}

\author{\IEEEauthorblockN{ Nikolas Hauschka\IEEEauthorrefmark{1}, Peter Balazs\IEEEauthorrefmark{1}\IEEEauthorrefmark{2}, and Lukas Köhldorfer\IEEEauthorrefmark{1}}
\IEEEauthorblockA{\IEEEauthorrefmark{1} \textit{Acoustics Research Institute}, \textit{Austrian Academy of Sciences}, Vienna, Austria.}
\IEEEauthorblockA{\IEEEauthorrefmark{2} \textit{Interdisciplinary Transformation University Austria (IT:U)}, Linz, Austria}
\emph{nikolas.hauschka@univie.ac.at,\{peter.balazs,lukas.koehldorfer\}@oeaw.ac.at}
\thanks{This work is supported by the FWF project (LoFT) P 34624 .}
}

\maketitle

\begin{abstract}
Associated with every separable Hilbert space $\Hil$ and a given localized frame, there exists a natural test function Banach space $\Hil^1$ and a Banach distribution space $\Hil^{\infty}$ so that $\Hil^1 \subset \Hil \subset \Hil^{\infty}$. In this article we close some gaps in the literature and rigorously introduce the space $\Hil^{\infty}$ and its weighted variants $\Hc_w^\infty$ in a slightly more general setting and discuss some of their properties. In particular, we compare the underlying weak$^*$- with the norm topology associated with $\Hc_w^\infty$ and show that $(\Hc_w^\infty, \Vert \cdot \Vert_{\Hc_w^\infty})$ is a Banach space.
\end{abstract}

\begin{IEEEkeywords}
Localized frame, co-orbit space, distribution space
\end{IEEEkeywords}

\section{Introduction}
The theory of distributions (or generalized functions) has become indispensable in modern mathematics, physics and engineering, and provides a suitable abstract framework for the analysis of various problems and the formalization of many phenomena. A classical example of a distribution space is the space $\mathcal{S}'(\mathbb{R}^d)$ of tempered distributions, i.e. the (anti-linear) topological dual of the space $\mathcal{S}(\mathbb{R}^d)$ of Schwartz functions on $\mathbb{R}^d$. Since $\mathcal{S}(\mathbb{R}^d) \subset L^2(\mathbb{R}^d)\subset \mathcal{S}'(\mathbb{R}^d)$ is 'the' prototypical Gelfand triple \cite{Antoine1998}. Since both $\mathcal{S}(\mathbb{R}^d)$ and $L^2(\mathbb{R}^d)$ are invariant under the Fourier transform, the space of tempered distributions is well suited for extending Fourier theory from $L^2(\mathbb{R}^d)$ to the distributional setting, and provides a valuable framework for analyzing various kind of PDEs \cite{Hörmander1990}. On the other hand, $\mathcal{S}(\mathbb{R}^d)$ is not a Banach space and $\mathcal{S}'(\mathbb{R}^d)$ not even a Fr\'echet space \cite{2021arXiv210611287B}, which sometimes makes working with these objects a bit tedious. If one is less oriented towards derivatives, but more interested in time-frequency analysis and applications to quantum physics, the \emph{Feichtinger algebra} $\mathbf{S}_0(\mathbb{R}^d)$ \cite{Feichtinger1981} and its dual space $\mathbf{S}'_0(\mathbb{R}^d)$, sometimes called the space of \emph{mild distributions} \cite{fei20}, serve as an indisputably useful alternative to the latter, see also \cite{Gosson21}. The space $\mathbf{S}_0(\mathbb{R}^d)$ is defined as the space of all elements in $L^2(\mathbb{R}^d)$, whose \emph{short-time Fourier transform} \cite{gr01} with respect to the Gaussian (or any other Schwartz function) is in $L^1(\mathbb{R}^{2d})$. In fact, $\mathbf{S}_0(\mathbb{R}^d)$ is not only a Banach space, but even a Banach algebra under both convolution and pointwise multiplication, and contains $\mathcal{S}(\mathbb{R}^d)$ as a norm-dense subspace. Thus $\mathbf{S}'_0(\mathbb{R}^d)$ is a Banach space of distributions contained in $\mathcal{S}'(\mathbb{R}^d)$. In particular, one can show that $\mathbf{S}_0(\mathbb{R}^d) \subset L^2(\mathbb{R}^d)\subset \mathbf{S}'_0(\mathbb{R}^d)$ (such a triple is called a \emph{Banach-Gelfand triple}), and since $\mathbf{S}_0(\mathbb{R}^d)$ is invariant under both the Fourier transform and actions of time-frequency shifts, $(\mathbf{S}_0(\mathbb{R}^d), \mathbf{S}'_0(\mathbb{R}^d))$ is widely considered as \emph{the} appropriate test function-/distribution space pair for time-frequency analysis and applications to quantum physics \cite{Gosson21}. It can be extended to a full range of space, the so-called modulation spaces \cite{gr01,fe06}.

In this article we consider a generalization of the Banach distribution space $\mathbf{S}'_0(\mathbb{R}^d)\supset L^2(\mathbb{R}^d)$ to the abstract Hilbert space setting via localized frames. In analogy to a characterization of $\mathbf{S}_0(\mathbb{R}^d)$ via Gabor frames \cite{gr01}, Fornasier and Gröchenig defined general co-orbit spaces $\Hil_w^p$ ($1\leq p\leq \infty$) associated with a localized frame in a given Hilbert space $\Hil$ \cite{forngroech1} to get a notion of 'nice' frames. This lead to interesting research \cite{gr03-4,futa09} and allowed the definition of co-orbit spaces \cite{daforastte08,fuvo15}, but without a necessary group structure.  In fact, for reasonable weights $w$, $\Hil^1_{1/w} \subset \Hil \subset \Hil_w^{\infty}$ is a Banach-Gelfand triple, which (up to norm equivalence) corresponds to the triple $\mathbf{S}_0(\mathbb{R}^d) \subset L^2(\mathbb{R}^d)\subset \mathbf{S}'_0(\mathbb{R}^d)$ in the case of certain Gabor frames \cite{forngroech1}. While the spaces $\Hil_w^p$ ($1\leq p <\infty$) and their properties have been studied in detail in e.g.\cite{forngroech1}, a rigorous discussion of the space $\Hil_w^{\infty}$ in \cite{forngroech1} was omitted by the authors \emph{"to avoid tedious technicalities"}. In the book chapter \cite{xxlgro14}, the authors gave a more detailed presentation of the space $\Hil_w^{\infty}$, still leaving some details to be filled. The purpose of this article is to close these gaps in the literature and give a detailed presentation of the space $\Hil_w^{\infty}$ and its topologies under minimal assumptions. To understand the properties of $\Hil^{\infty}_w$ in detail will be very helpful for e.g. describing operators between co-orbit spaces, in particluar elements from $\mathcal{B}(\Hil^{\infty},\Hil^1)$ \cite{bytchenkoff2023outer,feizim1}.     
\section{Preliminaries}
Let $\Hc$ be a separable $\C$-Hilbert space with inner product $\jbr{\cdot,\cdot}$ being conjugate linear in the second slot and let $\norm\cdot$ be the induced norm. The set $X$ is assumed to be countable and is used as the index set for the $\ell^p$ spaces. A weight $w$ is defined to be a map from $X$ to $\R_{>0}$. For a sequence $(\alpha_k)_{k \in X}$, the weighted $\ell^p$-norm is given by
$$\norm{(\alpha_k)_{k \in X}}_{\ell_w^p} := \norm{(\alpha_kw(k))_{k \in X}}_{\ell^p}.$$
The space $\ell_w^p$ consists of all sequences with a finite $\ell_w^p$-norm and is a Banach space. For the pointwise topology, we use the notation $\pw$.

\subsection{A certain locally convex topology}

Let $V$ be a $\C$-vector space with a countable Hamel basis. Let $\Tc$ be another $\C$-vector space and assume that $(\Tc,V)$ is a dual pair \cite{MESSERSCHMIDT20152018} with associated non-degenerate sesquilinear form $\jbr{\cdot, \cdot}_{\Tc, V}$ (being conjugate-linear in the second slot). We equip $\Tc$ with the Hausdorff locally convex $\sigma(\Tc, V)$-topology generated by the seminorms $\{|\jbr{\cdot, v}_{\Tc, V}|: v \in V\}$. Equivalently, we can replace $V$ by a Hamel basis, to get countably many seminorms generating the same topology, so that $\Tc$ is metrizable. Let $\overline \Tc$ denote the topological completion \cite{Merkle1998} of $\Tc$, which is again a metrizable Hausdorff locally convex space. Since this construction is fundamental for our setting we present some details. First, we make sure that the sesquilinear form $\jbr{\cdot, \cdot}_{\Tc, V}$ can be uniquely extended. For each $v \in V$, the linear functional
$$l_v:\Tc \to \C, \quad l_v(x) := \jbr{x, v}_{\Tc, V}.$$
is continuous, hence there exists a unique continuous linear extension $\overline l_v:\overline\Tc \to \C$ \cite[Theorem 5.1]{treves2016topological}. In particular, we may define the extended sesquilinear form $\jbr{\cdot,\cdot}_{\overline\Tc,V}$ on $\overline\Tc \times V$ by
$$\jbr{x,v}_{\overline\Tc,V} := \overline l_v(x).$$
We easily check for conjugate linearity on the second slot. 
Let $(x_n)_{n=1}^{\infty} \subset \Tc$ be a sequence converging to some $x \in \overline\Tc$. Then
$$\begin{aligned}
    \jbr{x,\lambda u+v}_{\overline\Tc, V} &= \lim_{n \to \infty}\jbr{x_n,\lambda u+v}_{\Tc, V}\\
    &= \overline\lambda\lim_{n \to \infty}\jbr{x_n, u}_{\Tc, V}+\lim_{n \to \infty}\jbr{x_n,v}_{\Tc, V}\\
    &= \overline\lambda\jbr{x, u}_{\overline\Tc, V}+\jbr{x,v}_{\overline\Tc, V}.
\end{aligned}$$
Finally, we make sure that our extended sesquilinear form is still non-degenerate. Suppose $\jbr{x,v}_{\overline\Tc,V} = 0$ for all $v \in V$. Let $(x_n)_{n=1}^{\infty}\subseteq \Tc$ converge to $x \in \overline\Tc$. Then 
$$0 = \jbr{x,v}_{\overline\Tc,V} = \lim_{n\to \infty} \jbr{x_n,v}_{\Tc, V} \quad (\forall v\in V),$$
so $(x_n)_{n=1}^{\infty}$ converges to $0$ with respect to $\sigma(\Tc,V)$, implying $x = 0$, since limits in Hausdorff spaces are unique.

\subsection{Frames}

A countable family $\psi = (\psi_k)_{k \in X}$ of vectors in $\Hil$ is called a \emph{frame}, if 
$$\sum_{k \in X} \abs{\jbr{f, \psi_k}}^2 \asymp \norm f^2 \qquad  (\forall f \in \Hc).$$
We refer to \cite{ole1n} for a comprehensive introduction to frame theory. Whenever $\psi$ is frame, both the \emph{coefficient operator} (or \emph{analysis operator}), defined as
$$C_{\psi} : \Hil \to  \ell^2 ,\quad C_\psi f= \big(\langle f   ,   \psi_k \rangle\big)_{k \in X}$$
and the \emph{synthesis operator}, defined as
\begin{equation*}
D_\psi : \ell^2 \to \Hil ,\quad D_\psi c =  \sum_{k \in X} c_k \psi_k,
\label{eq:Synthesenoperator}
\end{equation*}
are bounded and adjoint to one another, where the latter series converges unconditionally in $\Hil$. Additionally, $D_{\psi}$ is surjective. Their composition yields the \emph{frame operator}
$$S_\psi  = D_\psi C_\psi: \Hil \rightarrow \Hil ,\quad S_\psi f = \sum_{k \in X} \langle f   ,   \psi_k \rangle  \psi_k,$$
which is bounded, positive, self-adjoint and invertible. Composing the frame operator with its inverse (and vice versa) yields the \emph{frame reconstruction} formulae 
\begin{equation}\label{framerec}
f= \sum_{k \in X} \langle f, \tilde{\psi}_k \rangle  \psi_k  = \sum_{k \in X} \langle f, \psi_k \rangle  \tilde{\psi}_k \qquad  (\forall f\in \Hil).
\end{equation}
The family $\tilde{\psi} = \big(\tilde{\psi}_k \big)_{k \in X} := \big( S_{\psi}^{-1} \psi_k \big)_{k \in X}$ is a frame as well and called the \emph{canonical dual frame}. More generally, if $\psi^d = (\psi^d)_{k\in X}$ is another frame in $\Hil$ such that (\ref{framerec}) holds after replacing each $S_{\psi}^{-1} \psi_k$ with $\psi^d_k$, then $\psi^d$ is called a \emph{dual frame} of $\psi$. Finally, the \emph{cross Gram matrix} $G_{\tilde\psi,\psi}$ associated with two frames $\psi$ and $\tilde\psi$ is given by
$$G_{\tilde\psi,\psi}  = \big[ \langle \psi_{l} , \tilde\psi_k \rangle \big]_{k,l \in X}.$$
In fact, $G_{\tilde\psi,\psi}$ defines a bounded operator on $\ell^2$ and 
$$G_{\tilde\psi,\psi}=C_{\tilde\psi} D_\psi .$$

%
%
%
%
%

\section{Details on the topologies on $\Hil^{\infty}_w$}

In order to introduce the co-orbit spaces $\Hc_w^\infty$ we fix the following assumptions for the remainder of this article:
\begin{itemize}
    \item[(1)] $\psi$ is a frame for $\Hil$ and $\tilde\psi$ a dual frame.
    \item[(2)] The cross Gram matrix $G_{\tilde\psi,\psi}$ defines a bounded operator on $\ell_w^{\infty}$ for some fixed weight $w$.
\end{itemize}
We explicitly emphasize that conditions (1$-$2) are met for any weight $w$ whenever $\psi$ is a Riesz basis and $\tilde\psi$ its canonical dual Riesz basis (and in particular, when $\psi = \tilde\psi$ is an orthonormal basis). Further examples of pairs of dual frames $(\tilde\psi,\psi)$ and weights $w$ satisfying (1$-$2) are given by an intrinsically localized frame and its canonical dual frame, where $w$ is a so-called \emph{admissible weight}, see \cite{forngroech1} and the references therein. 
   
\

Now let $\Hc^{00} := \spanl (\tilde\psi_k)_{k \in X} = D_{\tilde\psi}(c_{00})$. Since $c_{00}$ a dense subspace of $\ell^2$ and $D_{\tilde\psi}:\ell^2 \to \Hil$ bounded and onto, we have that $\Hc^{00}$ is a dense subspace of $\Hil$.  Consequently, $(\Hil,\Hil^{00})$ is a dual pair with associated non-degenerate sesquilinear form $\langle \cdot, \cdot\rangle$ restricted to $\Hc \times \Hc^{00}$. Note, that the dual frame $\tilde\psi$ is a countable Hamel basis for $\Hil^{00}$. Next, let $\overline \Hc$ be the completion of $\Hc$ with respect to the $\sigma(\Hc, \Hc^{00})$-topology, with induced sesquilinear form $\jbr{\cdot, \cdot}_{\overline\Hc, \Hc^{00}}$. Then, we define $\Hc_w^\infty$ as the subspace of all $f \in \overline\Hc$, for which there exists a sequence $(f_n)_{n=1}^{\infty} \subseteq \Hc$ satisfying
$$\limn^{\sigma(\overline\Hc,\Hc^{00})}f_n = f \, \text{ and } \, \big\vert \jbr{f_n,\tilde\psi_k}\big\vert w(k) \lesssim 1 \, \,  (\forall k \in X, n \in \N).$$
Let $\jbr{\cdot, \cdot}_{\Hc_w^\infty, \Hc^{00}}$ be the restriction of $\jbr{\cdot, \cdot}_{\overline\Hc, \Hc^{00}}$ to $\Hil^{\infty}_{w} \times \Hil^{00}$. 

\subsection{The coefficient operator}

Having defined our space, we can easily define the coefficient operator with respect to $\tilde\psi$.

\begin{defin}
We define the coefficient operator as
$$C_{\tilde\psi}:\Hc_w^\infty \to \ell_w^\infty, \quad f \mapsto \big(\jbr{f,\tilde\psi_k}_{\Hc_w^\infty, \Hc^{00}}\big)_{k \in X}.$$ 
\end{defin}

By the definition of $\Hc_w^\infty$ this operator is well-defined and easily seen to be linear.

\begin{prop}
    The coefficient operator is injective.
\end{prop}

\begin{proof}
    Suppose $C_{\tilde\psi}f = 0$ for some $f \in \Hc_w^\infty$. Then $\jbr{f,\tilde\psi_k}_{\Hc_w^\infty, \Hc^{00}} = 0$ for all $k \in X$. Since the frame elements $\tilde\psi_k$ span $\Hc^{00}$, we must have $f = 0$, since $\jbr{\cdot, \cdot}_{\Hc_w^\infty, \Hc^{00}}$ is non-degenerate with respect to the first slot. 
\end{proof}


%
Since $\norm\cdot_{\ell_w^\infty}$ is a norm and the coefficient operator, as defined above, is injective, we immediately obtain the following.
%
%
%

\begin{cor}
The map 
$$\norm{\cdot}_{\Hc_w^\infty}: \Hc_w^\infty \to \R, \quad \norm f_{\Hc_w^\infty} := \big\Vert C_{\tilde\psi} f \big\Vert_{\ell_w^\infty}$$   
defines a norm on $\Hc_w^\infty$. In particular, $C_{\tilde\psi}$ is an isometry.
\end{cor}
%

\begin{prop}\label{normbounded}
    For every $f \in \Hc_w^\infty$ there exists a $\norm\cdot_{\Hc_w^\infty}$-bounded sequence $(f_n)_{n = 1}^\infty \subseteq \Hc_w^\infty\cap\Hc$ converging to $f$ with respect to $\sigma(\Hc_w^\infty, \Hc^{00})$.
    \label{Norm bounded sequence converging weakly*}
\end{prop}

\begin{proof}
    By definition, there exists a sequence $(f_n)_{n=1}^{\infty} \subset \Hc$ converging to $f$ with respect to $\sigma(\overline\Hc, \Hc^{00})$ satisfying
    $$\big\vert \jbr{f_n,\tilde\psi_k}\big\vert w(k) \lesssim 1 \quad  (\forall k \in X, \forall n \in \N).$$
    For each $n\in \mathbb{N}$, we choose the constant sequence $(f_n)_{m=1}^{\infty}$ to verify that $f_n \in \Hc_w^\infty$. Indeed, 
    $$\norm{f_n}_{\Hc_w^\infty} = \sup_{k\in X} \big \vert \jbr{f_n, \tilde\psi_k}\big\vert w(k) \lesssim 1$$
    for each $n\in \mathbb{N}$, as was to be shown.
\end{proof}

For completeness reason we show that the norm topology is stronger than the $\sigma(\Hc_w^\infty, \Hc^{00})$ topology.

\begin{thm}
    If a sequence $(f_n)_{n=1}^{\infty} \subset \Hc_w^\infty$ converges in norm, then it converges with respect to $\sigma(\Hc_w^\infty, \Hc^{00})$.
    \label{Normconvergence implies weak* convergence}
\end{thm}

\begin{proof}
    Without loss of generality, assume that $(f_n)_{n=1}^{\infty}$ converges to $0$. This means that
    $$\limn \big\Vert C_{\tilde\psi}f_n \big\Vert_{\ell_w^\infty} = 0.$$
    This implies for each $k \in X$ that
    $$\limn \big\vert \jbr{f_n,\tilde\psi_k} \big\vert = \limn|(C_{\tilde\psi}f_n)_k| = 0,$$
    as was to be shown.
\end{proof}

\begin{rem}
    The converse of the above statement is not true. Indeed, let $X = \N$, $w=1$ and $(e_n)_{n \in \N} = (\psi_n)_{n \in \N} = (\tilde\psi_n)_{n \in \N}$ be an orthonormal basis for $\Hil$. It is a standard routine to show that the sequence $(f_n)_{n=1}^{\infty} = (\psi_n)_{n=1}^{\infty}$ converges weakly in $\Hil$ (to $0$), hence also with respect to $\sigma(\Hc_w^\infty, \Hc^{00})$, but not in norm, since
    $$\limn\norm{e_n-0}_{\Hc_w^\infty} = \limn\sup_{k \in \N}|\jbr{e_n,e_k}_{\Hc_w^\infty, \Hc^{00}}| = 1.$$ 
\end{rem}

\subsection{Completeness}

Next, we show that $(\Hc_w^\infty, \Vert \cdot \Vert_{\Hc_w^\infty})$ is a Banach space. Before we are able to do so, we need a bit of work. 


\begin{lem}\label{psilinHinfty}
For each $l\in X$, $\psi_l \in \Hc_w^\infty$. 
\end{lem}

\begin{proof}
It suffices to show that $\big(\langle \psi_l, \tilde\psi_k \rangle \big)_{k \in X} \in \ell_w^\infty$ for each $l\in X$. This is indeed the case, since 
\begin{flalign}
\sup_{k\in X} \vert \langle \psi_l, \tilde\psi_k \rangle \vert w(k) &= w(l) \sup_{k\in X} \vert \langle \psi_l, \tilde\psi_k \rangle \vert w(k) w(l)^{-1} \notag \\
&\leq w(l) \sup_{k\in X} \sum_{m\in X} \vert \langle \psi_m, \tilde\psi_k \rangle \vert w(k) w(m)^{-1} \notag \\
&= w(l) \big\Vert G_{\tilde\psi,\psi} (w(m)^{-1})_{m\in X} \big\Vert_{\ell^{\infty}_w} \notag \\
&= w(l) \big\Vert G_{\tilde\psi,\psi} \big\Vert_{\B(\ell^{\infty}_w)} < \infty, \notag
\end{flalign}
where we used that $\Vert (w(m)^{-1})_{m\in X} \Vert_{\ell^{\infty}_w} =1$. 
\end{proof}

Note that the latter shows that $\jbr{\psi_l,\tilde\psi_k}$ and $\jbr{\psi_l, \tilde\psi_k}_{\Hc_w^\infty, \Hc^{00}}$ are the same for all $k,l\in X$.



\begin{lem}
    Let $(f_n)_{n = 1}^\infty \subseteq \Hc\cap \Hc_w^\infty$ be a $\Hc^\infty$-norm bounded sequence and fix $k \in X$. Then the sequence
    $$\brac{\jbr{f_n, \tilde\psi_l}\jbr{\psi_l, \tilde\psi_k}}_{l \in X}$$
    is dominated by an $\ell^1$ sequence.
    \label{Dominated by l^1}
\end{lem}

\begin{proof}
    Recall from the definition of $\Hil^{\infty}_w$ that
    $$\big\vert\jbr{f_n,\tilde\psi_l} \big\vert w(l) \leq \big\Vert C_{\tilde\psi}f_n \big\Vert_{\ell_w^\infty} = \norm{f_n}_{\Hc_w^\infty} \lesssim 1 \quad \forall n \in \N.$$
    This implies that
    $$\begin{aligned}\big\vert \jbr{f_n,\tilde\psi_l} \big\vert \big\vert\jbr{\psi_l,\tilde\psi_k} \big\vert &= \big\vert \jbr{f_n,\tilde\psi_l} \big\vert  w(l) \big\vert\jbr{\psi_l,\tilde\psi_k}\big\vert w(l)^{-1}\\
    &\lesssim \big\vert \jbr{\psi_l,\tilde\psi_k} \big\vert w(l)^{-1}.
    \end{aligned}$$
To show that the latter is in $\ell^1$ (with respect to $l$), we estimate similarly as in the proof of Lemma \ref{psilinHinfty}: 
\begin{flalign}
&\sum_{l\in X} \big\vert \jbr{\psi_l,\tilde\psi_k} \big\vert w(l)^{-1} \notag \\
&\leq w(k)^{-1} \sup_{m\in X} \sum_{l\in X} \big\vert \jbr{\psi_l,\tilde\psi_m} \big\vert w(m) w(l)^{-1} \notag \\ 
&\leq w(k)^{-1} \big\Vert G_{\tilde\psi,\psi} \big\Vert_{\Bc(\ell_w^\infty)}.\notag
\end{flalign}
\end{proof}

Now we are able to show that the Gramian matrix is the identity on the range of $C_{\tilde\psi}:\Hil_w^{\infty} \to \ell_w^{\infty}$.

\begin{thm}
    It holds $G_{\tilde\psi, \psi}|_{R(C_{\tilde\psi})} = \id_{R(C_{\tilde\psi})}$.
    \label{Restricted Gramian is identity}
\end{thm}

\begin{proof}
Let $f \in \Hc_w^\infty$. By Proposition \ref{normbounded}, there exists a $\Hc_w^\infty$-norm bounded sequence $(f_n)_{n = 1}^\infty \subseteq \Hc\cap \Hc_w^\infty$ such that 
    $$\limn^{\sigma(\Hc^\infty,\Hc^{00})} f_n = f.$$
    Our goal is to show that
    $$(G_{\tilde\psi,\psi}C_{\tilde\psi}f)_k = (C_{\tilde\psi}f)_k \quad \forall k \in X.$$
    The idea is to replace $f$ by the limit of the sequence $(f_n)_{n = 1}^\infty$, then apply dominated convergence to swap the limit with the sum. Finally, we apply the frame reconstruction formula (\ref{framerec}) to $f_n$. To make the computation simpler, we start from the inside and move step by step to the outside. First, note that
    $$(C_{\tilde\psi}f)_l = \jbr{f, \tilde\psi_l}_{\Hc_w^\infty, \Hc^{00}} = \limn\jbr{f_n, \tilde\psi_l}.$$
    Since $\psi_l \in \Hc^{\infty}_w$ by Lemma \ref{psilinHinfty}, this yields
    $$(C_{\tilde\psi}f)_l\psi_l = \brac{\limn\jbr{f_n,\tilde\psi_l}}\psi_l = \limn^{\Hc_w^\infty}\brac{\jbr{f_n,\tilde\psi_l}\psi_l}.$$
    Then
    $$\begin{aligned}\jbr{(C_{\tilde\psi}f)_l\psi_l, \tilde\psi_k}_{\Hc_w^\infty, \Hc^{00}} &= \ljbr{\limn^{\Hc_w^\infty}\brac{\jbr{f_n,\tilde\psi_l}\psi_l}, \tilde\psi_k}_{\Hc_w^\infty,\Hc^{00}}\\
    &= \limn\jbr{\jbr{f_n,\tilde\psi_l}\psi_l,\tilde\psi_k}\\
    &= \limn \jbr{f_n, \tilde\psi_l}\jbr{\psi_l, \tilde\psi_k}.
    \end{aligned}$$
    Now, observe that
    $$\begin{aligned}(G_{\tilde\psi,\psi}C_{\tilde\psi}f)_k &= \sum_{l\in X} \langle \psi_l, \tilde\psi_k \rangle (C_{\tilde\psi}f)_l \\    
    &= \sum_{l \in X}\jbr{(C_{\tilde\psi}f)_l\psi_l,\tilde\psi_k}_{\Hc_w^\infty, \Hc^{00}}\\
    &= \sum_{l \in X}\limn \jbr{f_n, \tilde\psi_l}\jbr{\psi_l, \tilde\psi_k} = (*).
    \end{aligned}$$
    By Lemma \ref{Dominated by l^1}, $\big(\jbr{f_n, \tilde\psi_l}\jbr{\psi_l, \tilde\psi_k}\big)_{l \in X}$ is dominated by an $\ell^1$-sequence for each $n \in \N$, so we can interchange the sum and the limit using dominated convergence and obtain
    $$\begin{aligned}(*)&=\limn\sum_{l \in X} \jbr{f_n, \tilde\psi_l}\jbr{\psi_l, \tilde\psi_k}\\
    &= \limn \ljbr{\underbrace{\sum_{l \in X}\jbr{f_n,\tilde\psi_l}\psi_l}_{= f_n \text{ (frame reconstruction)}},\tilde\psi_k}\\
    &= \limn\jbr{f_n,\tilde\psi_k} = \jbr{f,\tilde\psi_k}_{\Hc_w^\infty, \Hc^{00}} = (C_{\tilde\psi}f)_k.\end{aligned}$$
    This shows that $G_{\tilde\psi,\psi} C_{\tilde\psi} f = C_{\tilde\psi} f$ for all $f \in \Hc_w^\infty$.
\end{proof}

Next, we show that the range of the coefficient operator is the only set on which the Gramian matrix acts as the identity.

\begin{lem}
    Let $\alpha \in \ell_w^\infty$ such that $G_{\tilde\psi,\psi}\alpha = \alpha$. Then $\alpha \in R(C_{\tilde\psi})$.
\end{lem}

\begin{proof}
    Choose a nested sequence $(F_n)_{n\in\N}$ of finite subsets of $X$ such that $\bigcup_{n \in \N} F_n = X$. Let
    $$f_n := \sum_{l \in F_n} \alpha_l \psi_l \in \Hc \cap \Hc_w^\infty.$$
    Observe that for each $k\in X$
    \begin{flalign}\label{XX}
        \big\vert \jbr{f_n, \tilde\psi_k}\big\vert w(k) &= \abs{\ljbr{\sum_{l \in F_n}\alpha_l\psi_l,\tilde\psi_k}}w(k)\notag \\
        &\leq \sum_{l \in F_n}|\alpha_l|\big\vert \jbr{\psi_l,\tilde\psi_k} \big\vert w(k)\notag\\
        &\leq \sum_{l \in X}|\alpha_l|w(l)w(l)^{-1}\big\vert \jbr{\psi_l,\tilde\psi_k} \big\vert w(k)\notag\\
        &\leq \norm\alpha_{\ell_w^\infty}\sum_{l \in X}w(l)^{-1} \big\vert \jbr{\psi_l,\tilde\psi_k}\big\vert w(k)\\
        &\leq \norm\alpha_{\ell_w^\infty} \big\Vert G_{\tilde\psi,\psi} \big\Vert_{\Bc(\ell_w^{\infty})}.\notag
    \end{flalign}
    This implies that $(f_n)_{n = 1}^\infty$ is $\Hc_w^\infty$-bounded. Since the sum in (\ref{XX}) converges absolutely, we can also easily deduce, that $(f_n)_{n = 1}^\infty$ is a Cauchy sequence with respect to $\sigma(\Hc_w^\infty,\Hc^{00})$, since for $m \geq n$
    $$\begin{aligned}\big\vert \jbr{f_m-f_n, \tilde\psi_k}\big\vert w(k) &\lesssim \sum_{l \in F_m \setminus F_n}w(l)^{-1}\big\vert \jbr{\psi_l,\tilde\psi_k} \big\vert w(k)\\
    &\leq \sum_{l \in X \setminus F_n}w(l)^{-1}\big\vert\jbr{\psi_l,\tilde\psi_k} \big\vert w(k)\\    &\stackrel{n\to\infty}\to 0.
    \end{aligned}$$
    Also,
    $$\begin{aligned}
        \alpha_k = (G_{\tilde\psi,\psi}\alpha)_k &= \lim_{n\to \infty} \sum_{l\in F_n} \alpha_l \langle \psi_l,\tilde\psi_k\rangle \\
        &= \limn\ljbr{\sum_{l \in F_n}\alpha_l\psi_l,\tilde\psi_k}\\
        &= \limn \jbr{f_n,\tilde\psi_k}.
    \end{aligned}$$
    Now let
    $$f := \limn^{\sigma(\Hc_w^\infty, \Hc^{00})} f_n \in \Hc_w^\infty.$$
    Then, we see that
    $$\begin{aligned}
        C_{\tilde\psi}f = \brac{\jbr{f, \tilde\psi_k}}_{k \in X} &= \brac{\limn\jbr{f_n, \tilde\psi_k}}_{k \in X}\\
        &= (\alpha_k)_{k \in X} = \alpha.
    \end{aligned}$$
    Thus we conclude that $\alpha \in R(C_{\tilde\psi})$.
\end{proof}

From the last two results, we obtain the following theorem:

\begin{thm}\label{isom}
    Let $V:= \{\alpha \in \ell_w^\infty: \alpha = G_{\tilde\psi,\psi}\alpha\}$. Then $C_{\tilde\psi}:\Hil^{\infty}_w \to V$ is an isometric isomorphism.
\end{thm}

Finally, we can show the completeness of $\Hc_w^\infty$.

\begin{thm}
    $V$ is a closed subspace of $\ell_w^\infty$. Consequently, $\Hc_w^\infty$ is a Banach space.
\end{thm}

\begin{proof}
    Let $(\alpha^n)_{n = 1}^\infty \subseteq V$ be a sequence converging to some $\alpha \in \ell_w^\infty$. Since $G_{\tilde\psi,\psi}$ is bounded on $\ell^{\infty}_{w}$, we get
    $$G_{\tilde\psi,\psi}\alpha = G_{\tilde\psi,\psi}\limn^{\ell_w^\infty}\alpha^n = \limn^{\ell_w^\infty}G_{\tilde\psi,\psi}\alpha^n = \limn^{\ell_w^\infty}\alpha^n = \alpha.$$
    This implies that $V$ is a closed subspace of $\ell_w^\infty$. Since, by Theorem \ref{isom}, $V$ is isometrically isomorphic to $\Hc_w^\infty$, the latter is complete as well.
\end{proof}


\bibliographystyle{abbrv}
\bibliography{biblioall}

\end{document}